\documentclass[12pt, reqno]{amsart}
\usepackage{amsmath, amsthm, amscd, amsfonts, amssymb, graphicx, color}
\usepackage[bookmarksnumbered, colorlinks, plainpages]{hyperref}

\makeatletter \oddsidemargin.9375in \evensidemargin \oddsidemargin
\def\authorsaddresses#1{\dedicatory{#1}}
\newtheorem{theorem}{Theorem}[section]
\newtheorem{lemma}[theorem]{Lemma}

\theoremstyle{definition}

\theoremstyle{remark}

\numberwithin{equation}{section}

\begin{document}
\setcounter{page}{1}

%%%%%%%%%%%%%%%%%%%%%%%%%%%%%%%%%%%%%%%%%%%%%%%
%% Please do not remove the following statement.
%%%%%%%%%%%%%%%%%%%%%%%%%%%%%%%%%%%%%%%%%%%%%%%

%%%%%%%%%%%%%%%%%%%%%%%%%%%%%%%%%%%%%%%%%%%%%%%
%% ==>start from here
%% Start from here
%%%%%%%%%%%%%%%%%%%%%%%%%%%%%%%%%%%%%%%%%%%%%%%

%%%%%%%%%%%%%%%%%%%%%%%%%%%%%%%%%%%%%%%%%%%%%%%%%%%%%%%%%%%%%%%%%%%%%
% Insert title of your article. Note: \title[short title]{full title}
%%%%%%%%%%%%%%%%%%%%%%%%%%%%%%%%%%%%%%%%%%%%%%%%%%%%%%%%%%%%%%%%%%%%%
\title[A bootstrap approximation to $L_p$-statistic in length-biased ]{A bootstrap approximation to $L_p$-statistic of kernel density estimator in length-biased model  
 }
%%%%%%%%%%%%%%%%%%%%%%%%%%%%%%%%%%%%%%
% Author's name must be inserted here
%%%%%%%%%%%%%%%%%%%%%%%%%%%%%%%%%%%%%%

\author[ Zamini]{ Raheleh Zamini}

%%%%%%%%%%%%%%%%%%%%%%%%
% Addresses
%%%%%%%%%%%%%%%%%%%%%%%%
\authorsaddresses{  Faculty of Mathematical Sciences and  Computer, Department  of Mathematics, Kharazmi University, Thehran, Iran. \\
zamini@khu.ac.ir}
%\vspace{0.5cm} $^2$ Department of  Mathematics, Ferdowsi
%University of
%Mashhad, Mashhad, Iran.\\
%rahelehzamini@yahoo.com}
%%%%%%%%%%%%%%%%%%%%%%%%%%%%
% Subject class; see http://www.ams.org/mathscinet/msc/msc2010.html
%%%%%%%%%%%%%%%%%%%%%%%%%%%%
%\subjclass[2010]{Primary 47A55; Secondary 39B52, 34K20, 39B82.}

%%%%%%%%%%%%%%%%%%%%%%%%%%%%%%%%%%%%%%%%%%%%%%%%%%%%%%%%%%%%%%%%%%%%%%%%%%%%%%%%%%%%%%%%%%
% keywords; Note that the number of keywords must be at least 3 items and at most 5 items.
%%%%%%%%%%%%%%%%%%%%%%%%%%%%%%%%%%%%%%%%%%%%%%%%%%%%%%%%%%%%%%%%%%%%%%%%%%%%%%%%%%%%%%%%%%
%%%%%%%%%%%%%%%%%%%%%%%%%%%%%%%%%%%%%%%%%%%%%%%%%%%%%%%%%%%%%%%%%%%%%%%%%%%%%%%%%%%%%%%%
% Please do not change the following statements. You must start from line 43 "==>start from here"
%%%%%%%%%%%%%%%%%%%%%%%%%%%%%%%%%%%%%%%%%%%%%%%%%%%%%%%%%%%%%%%%%%%%%%%%%%%%%%%%%%%%%%%%
\maketitle

\begin{abstract}
%In this article we present an approximation to the error quantity $I_{n}(p)=\int_{\mathbb{R}} {\Big|f_{n}(x)-f(x)\Big|}^{p} d\mu(x),\quad 1\leq p<\infty,$ where $f_{n}$ is the kernel density estimator proposed by Jones(1991) for length-biased data. The approach is based on the invariance principle for the empirical processes proved by Horv\'{a}th (1985). We study the difference $I_{n}(p)$ with its estimator in terms of its rates of convergence to zero. A central limit theorem for estimator of $I_{n}(p)$ will then follow immediately.
  This article presents a bootstrap approximation to the eroor quantity $I_n(p)=\int_{\mathbb{R}}{\big|f_n(x)-f(x)\big|}^p d\mu(x), \leq p<\infty,$ where $f_n$ is  the kernel density estimator proposed by Jones \cite{11} for length-biased data. The article establishes one bootstrap central limit theorem for the corresponding bootstrap version of $I_n(p).$ 
  %The method is illustrated by real automobile brake pads data.
 \end{abstract}
\section{Introduction}
Kernel methods are widely used to estimate probability density functions. Suppose $X_i \in \mathbb{R}$ are independent with a common density $f$. Then the kernel density estimate of $f$ is 
\begin{eqnarray*}
f_{n}(x)=\frac{1}{nh_n}\sum_{i=1}^{n}K\big(\frac{x-X_i}{h_n}\big),
\end{eqnarray*}
where, $K$ is a kernel function and $h_n$ is a sequence of (positive) "bandwidths" tending to zero as $n\longrightarrow \infty.$ (See Rosenblatt  \cite{15}). A standard and common measure of $f_n$ is given by the $L_p$ distance 
\begin{eqnarray*}
I_{n}(p)=\int_{\mathbb{R}}{\big| f_{n}(x)-f(x)\big|}^{p} d\mu(x), 1\leq p<\infty,
\end{eqnarray*}
where, $\mu$ is a measure on the Borel sets of $\mathbb{R}$. The mean integrated square error, that is, $E(I_{n}(2))$, is a very popular measure of the distance of $f_n$ from $f$. The other well-investigated case is $p=1.$ In general, $I_{n}(p)$ can be used to carry out tests of hypothesis (asymptotic) for the density $f$. Cs\"{o}rg\H{o}  and Horv\'{a}th \cite{4} obtained a centeral limit theorem for $L_p$ distances $(1 \leq p <\infty)$ of kernel estimators based on complete samples. In the random censorship model, Cs\"{o}rg\H{o} et.al. \cite{2} obtained central limit theorems for $L_p$ distances $(1 \leq p<\infty)$ of kernel estimators. They tested their result in Monte Carlo trails and applied them for goodness of fit. Groeneboom et. al. \cite{8} studied the asymptotic normality of a suitably rescaled version of the $L_1$ distance of the Grenander estimator, using properties of a jump process was introduced by Gronenboom \cite{7}.
In Length-biased setting, Fakoor and Zamini \cite{6}, proved a central limit theorem for $L_p$ distances $(1 \leq p<\infty).$ Also they presented a central limit theorem for approximation of $I_{n}(p).$ Mojirsheibani \cite{14}, presented two approximations for  $L_p$ distances $(1 \leq p<\infty)$ on complete samples. He also defined two approximations for $I_{n}(p)\quad (1 \leq p<\infty)$ based on bootstrap versions of $I_{n}(p)$ with central limit theorems for them.

In this paper a bootstrap version of $I_{n}(p)$ in Length-biased sampling is defined and a central limit theorem for it is defined. In biased sampling, the data are sampled from a distribution different from censoring sampling. In censoring, some of the observations are not completely observed, but are known only to belong to a set. The prototypical example is the time until an event. For an event that has not happened by time t, the value is known only to be in $(t,\infty).$ Truncation is a more severe distortion than censoring. Where censoring replaces a data value by a subset, truncation deletes that value from the sample if it would have been in a certain range. Truncation is an extreme form of biased sampling where certain data values are unobservable. Length- biased data appear naturally in many fields, and particularly in problems related to renewal processes. This special truncation model has been studied by e.g. Wicksel \cite{17}, McFadden \cite{13}, Cox \cite{1}, Vardi \cite{16}.  
\section{Main results}
Let $Y_1,\ldots, Y_n$ be $n$ independent and identically distribute (i.i.d.) nonnegative random variables (r.v.) from a distribution $G,$ defined on ${\mathbb{R}}^{+}=[0,\infty).$ $G$ is called a length-biased distribution corresponding to a given distribution $F$ (also defined on ${\mathbb{R}}^{+})$, if 
\begin{eqnarray*}
G(t)={\mu}^{-1}\int_{0}^{t}xdF(x),\quad for \quad every \quad y \in {\mathbb{R}}^{+},
\end{eqnarray*}
where $\mu=v^{-1}=\int_{0}^{\infty} xdF(x)$ is assumed to be finite. A simple calculation shows that
\begin{eqnarray*}
F(t)=\mu \int_{0}^{t}y^{-1}dG(y), \quad t>0.
\end{eqnarray*}
The empirical estimatior of $F$ can be written by $G{n}(t)=\frac{1}{n}\sum_{i=1}^{n}I(Y_i \leq t),$ namely 
\begin{eqnarray*}
F_{n}(t)={\mu}_{n}\int_{0}^{t}{y}^{-1}dG_{n}(y),
\end{eqnarray*}
where ${\mu}_{n}^{-1}=v_{n}=\int_{0}^{\infty}y^{-1}dG_{n}(y).$
Using $F_{n}(t),$  the following estimator for density function of $f=F'$  in length-biased model is known, 
\begin{eqnarray}\label{dsetare}
f_{n}(t)=\frac{1}{h_n}\int_{\mathbb{R}} K\big(\frac{t-x}{h_n}\big) dF_{n}(x).
\end{eqnarray} (For some refrences about this subject, see the refrences are given in Fakoor and Zamini \cite{6}.)

We start by stating the further notations.
Assume that $T<\tau=\sup\big\{x, G(x)<1\big\}<\infty.$ Throught this paper $N=N(0,1)$ stands for a standard normal r.v. Let 
 \begin{eqnarray}\label{mtp}
%r(t)&=&\frac{\int_{\mathcal{R}}K(u)K(t+u)du}{\int_{\mathcal{R}}K^2(u) du},\\
m(p)=E{|N|}^{p}{\big(\int_{\mathcal{R}}K^{2}(t) dt\big)}^{p/2}\int_{0}^{T}f^{\frac{p+2}{2}} dt,
\end{eqnarray}
and
 \begin{eqnarray}\label{mttp}
{\sigma}^2={\sigma}_{1}^{2}\int_{0}^{T}f^{p+2}(t) dt{\big(\int_{\mathcal{R}}K^2(t) dt\big)}^{p},
\end{eqnarray}
where,
   \begin{eqnarray*}
  { \sigma}^{2}_{1}&=&{(2\pi)}^{-1} \int_{-\infty}^{+\infty}\bigg(\int_{-\infty}^{+\infty}\int_{-\infty}^{+\infty}{|xy|}^{p}{\big.(1-r^{2}(u)\big.)}^{-1/2}.\\&exp&\big(-\frac{1}{2(1-r^{2}(u))}\big(x^2-2xy r(u) +y^2\big)\big)dx dy-{(E{|N|}^{p})}^{2}\bigg) du,
\end{eqnarray*}
with
\begin{eqnarray*}
r(t)=\frac{\int_{\mathcal{R}}K(u)K(t+u)du}{\int_{\mathcal{R}}K^2(u) du}.
\end{eqnarray*}
Let $B(t,n)$ is a two-parameter Gaussian process with zero mean and covariance function
\begin{eqnarray*}
E\big[B(x, n)B(y,m)\big] = (mn)^{−1/2}
(m ∧ n)\big[G(x \wedge y) − G(x)G(y)\big ] (a \wedge b = \min(a, b)).
\end{eqnarray*}
Based on $B(x,n)$, Horv\'{a}th \cite{9}, defined the mean zero Gaussian process
  \begin{eqnarray}\label{Gamma}
 \Gamma(t,n)=\mu \int_{0}^{t}y^{-1}dB(y,n)-\mu F(t)\int_{0}^{\infty}y^{-1}dB(y,n),
 \end{eqnarray}
 with covariance function
 \begin{eqnarray}\label{EG}
 E\big(\Gamma(x,n)\Gamma(y,m)\big)&=&{(mn)}^{-1/2}(m \wedge n) \big[\sigma( x \wedge y)\nonumber\\
 &\quad -&F(x)\sigma(y)-F(y)\sigma(x)+F(x)F(y)\sigma \big],
  \end{eqnarray}
  such that $\Gamma(t,n)$ approximates the empirical process ${\alpha}_{n}(t)=\sqrt{n}\big[F_n(t)-F(t)\big].$
  In \eqref{EG}, $\sigma(t)={\mu}^{2}\int_{0}^{t}y^{-2}dG(y),$\quad and  $\sigma={\lim}_{t\longrightarrow \infty}\sigma(t)={\mu}^{2}\int_{0}^{\infty}y^{-2} dG(y)$.
 Fakoor and Zamini \cite{6}  used the strong approximation defined in \eqref{Gamma} and investigated asymptotic normality behavior 
\begin{eqnarray}\label{In}
   I_{n}(p)=\int_{0}^{T}{\big|f_{n}(x)-f(x)\big|}^{p}{\big(\frac{x}{\mu}\big)}^{p/2} dF(x),
     \end{eqnarray}
     where $f_n$ is defined in \eqref{dsetare}.
They showed that under some conditions on $h_n,$
 \begin{eqnarray}\label{Indistri}
  {(h_n {\sigma}^{2}(p))}^{-1/2}\big\{{(nh_n)}^{p/2}I_{n}(p)-m(p)\big\}\stackrel{D}{\longrightarrow}N(0,1).
  \end{eqnarray}
 In this article we prove one bootstrap central limit theorem for the corresponding bootstrap version of $I_n(p)$ in \eqref{In}.
 
 Given the random sample, $Y_1,\ldots,Y_n,$ let $Y_{1}^{*},\ldots,
 Y_{n}^{*}$ be a bootstrap sample drawn from $Y_1,\ldots, Y_n$. That is, $Y_{1}^{*},\ldots, Y_{n}^{*}$ be conditionally independent random variables with common distribution function $G_{n}(t)=\frac{1}{n}\sum_{i=1}^{n}I\big(Y_i \leq t\big).$ (See e.g. Cs\"{o}rg\H{o} et.al.\cite{5}). Let
 \begin{eqnarray*}
F_{n,n}(t)={\mu}_{n,n} \int_{0}^{t} y^{-1}dG_{n,n}(y),\\
%v_{n,n}&=&\int_{0}^{\tau}y^{-1}dG_{n,n}(y),
\end{eqnarray*}
and 
 \begin{eqnarray}\label{meli2}
f_{n,n}(x)=h_n^{-1}\int_{\mathbb{R}} K \Big(\frac{x-y}{h_n}\Big) dF_{n,n}(y),
\end{eqnarray}
where
\begin{eqnarray*}
{\mu}_{n,n}^{-1}=v_{n,n}=\int_{0}^{\tau}y^{-1}dG_{n,n}(y),
\end{eqnarray*}
with $G_{n,n}(y)=\frac{1}{n}\sum_{i=1}^nI\big[Y_i^* \leqslant n \big].$ Using $f_n$ and  $f_{n,n}(x)$ in \eqref{dsetare} and \eqref{meli2} respectively,  one can write the bootstrap version of $I_{n}(p)$ in \eqref{In} by
\begin{eqnarray}\label{Inn}
   I_{n,n}(p)=\int_{0}^{T}{\big| f_{n,n}(x)-f_{n}(x)\big|}^{p}{\big(\frac{x}{{\mu}_{n}}\big)}^{p/2} f_{n}(x) dx.
     \end{eqnarray}
A result  of Cs\"{o}rg\H{o} et al.\cite{5}, shows that there exists a sequence of  Brownian bridges  $\big\{B_{n,n}(t), 0\leq t \leq 1\big\}$ such  that
\begin{eqnarray}\label{eeeqII}
\sup_{-\infty <x<\infty} \Big |\beta_{n,n}(x)-B_{n,n}(G(x)) \Big |=O(n^{\frac{-1}{2}} \log n)\quad a.s.,
\end{eqnarray}
where
\begin{eqnarray}\label{eeeqIII}
\beta_{n,n}(x)=n^{\frac{1}{2}}\big[G_{n,n}(x)-G_n(x)\big].
\end{eqnarray}
Set
\begin{eqnarray}
\alpha_{n,n}(t)=n^{\frac{1}{2}}\big[F_{n,n}(t)-F_n(t)\big].
\end{eqnarray}
Clearly
\begin{eqnarray}\label{eeeq1}
\alpha_{n,n}(t)={\mu}_{n} \int _0^ty^{-1}d\beta_{n,n}(y)
-\Big(\int _0^ty^{-1}dG_{n,n}(y)\Big){\mu}_{nn}{\mu}_n \int _0^{\tau}y^{-1}d \beta_{n,n}(y).
\quad
\end{eqnarray}
This form of $\alpha_{n,n}$ suggests that the approximation processes for $\alpha_{n,n}$ will be
\begin{eqnarray}\label{eeeq2}
\Gamma_{n,n}(t)={\mu}_{n}\int _0^ty^{-1}dB_{n,n}(G(y))-
\Big(\int _0^ty^{-1}dG_{n,n}(y)\Big){\mu}_{nn}{\mu}_n \int _0^{\tau}y^{-1}d B_{n,n}(G(y)).
\quad
\end{eqnarray}
   In this article we use the approximation defined in \eqref{eeeq2} for ${\alpha}_{nn}$ and investigate asymptotic normality  behavior $I_{n,n}$ in \eqref{Inn}.
     The bootstrap is a widely used tool in statistics and, therefore, the properties of $I_{n,n}(p)$  are of great interest in applied as well as in theoretical statistics.
     
     Before stating our result, we list all assumptions used in this paper.\\
     \textbf{Assumptions}\\
$\bold C(1).$ $d\mu(t)=w(t) dt,$ where $w(t)\geq 0$ and continuous on $[0,\tau],$ where $T<\tau<\infty$ and $\tau=\sup\{x, G(x)<1\}.$\\
%$\bold K.$ Assumptions on the kernel K:\\
 $\bold K(1).$ There is a finite interval such that $K$ is continuous and bounded on it and vanishes outside of this interval.\\
$\bold K(2).$ $\int_{\mathbb{R}} K^{2}(t) dt >0.$\\
%$\bold K(3).$ $K$ is of bounded variation.\\
$\bold K(3).$ $K^{'}$ exists and is bounded.\\
$\bold K(4).$ $\int K(t) dt=1$.\\
%$\bold K(5).$ $\int_{\mathbb{R}} K(t) dt=1.$\\
%$\bold K(6).$ $\int_{\mathbb{R}} t K(t)dt=0,$\quad and $\int_{\mathbb{R}} {|t|}^{3} K(t)dt<\infty.$\\
%Assumptions on the density $f$:\\
%$\bold F(1).$ $f$ is uniformly bounded (a.s.) on the $(0,\tau).$\\
$\bold F(1).$ $\Big|\frac{f^{'}(x)}{x^{\frac{1}{2}}f^{\frac{1}{2}}(x)}\Big|$ and $\Big | \frac{f^{\frac{1}{2}}(x)}{x^{\frac{3}{2}}}\Big |$ is uniformly bounded (a.s.) on the $(0,\tau).$\\
%$\bold F(3).$ $f^{'''}$ exists and is uniformly bounded (a.s.) on the $(0,\tau).$\\
%Assumptions on the distributions function $G$:\\
$\bold G(1).$ $ {(G(x))}^{\frac{1}{r}}x^{-2}$ is uniformly bounded (a.s.) on the $(0,\tau)$ for some $r> 4$ .\\

  Define $\hat{m}(p)$ and ${\hat{\sigma}}^{2}(p)$ to be the counterparts of $m(p)$ and ${\sigma}^{2}(p),$ after replacing $f$ by $f_{n}$ in \eqref{mtp} and  \eqref{mttp} respectively, i.e.,
\begin{eqnarray}\label{fad20}
\hat{m}(p)&=& E{\big|N\big|}^{p}{\bigg(\int_{\mathbb{R}}K^{2}(t) dt\bigg)}^{\frac{p}{2}}\int_{0}^{T} f^{\frac{p+2}{2}}_{n}(t) dt,\nonumber\\
{\hat{\sigma}}^{2}(p)&=&{\sigma}^{2}_{1}\int_{0}^{T}f^{p+2}_{n}(t) dt{\bigg(\int_{\mathbb{R}}K^{2}(t) dt\bigg)}^{p}.
\end{eqnarray}
\begin{theorem}\label{fad23}
Suppose that Assumptions $\bold{K(1)}$-$\bold{K(4)}$, $\bold{C(1)}$, $\bold{F(1)}$ and $\bold{G(1)}$  hold. If, as $n\rightarrow \infty$
\begin{eqnarray*}
h_{n}\rightarrow 0, \quad \frac{\log \log n }{nh^{3}_n}\rightarrow0,\quad
n^{-\frac{1}{r}} h^{-1}_n\rightarrow 0,\quad \frac{\log n}{nh^{2}_n}\rightarrow0,
\end{eqnarray*}
then for $\hat{m}(p)$ and $\hat{\sigma}(p)$ in \eqref{fad20}, one has
\begin{eqnarray*}
{\big(h_n {\hat{\sigma}}^{2}(p)\big)}^{-\frac{1}{2}}\bigg\{{\big(nh_n\big)}^{\frac{p}{2}} I_{n,n}(p)-\hat{m}(p)\bigg\}\stackrel{D}{\longrightarrow} N(0,1),\quad 1<p<\infty.
\end{eqnarray*}
\end{theorem}
\begin{proof}
At first, notice that with using the inequality
\begin{eqnarray*}
\big|{|a(x)|}^{p}-{|b(x)|}^{p}\big|&\leq& p2^{p-1}{\big|a(x)-b(x)\big|}^{p}+p2^{p-1}{|b(x)|}^{p-1}\big|a(x)-b(x)\big|,\quad
\end{eqnarray*}
$($for $p\geq 1$$),$ and \eqref{do}, we may write
\begin{eqnarray}\label{2}
\big|{v_n}^{\frac{p}{2}}-{v}^{\frac{p}{2}}\big|&\leq& \frac{p}{2}2^{\frac{p}{2}-1}{\big|v_n-v\big|}^{\frac{p}{2}}+\frac{p}{2}2^{\frac{p}{2}-1}{|v|}^{\frac{p}{2}-1}\big|v_n-v\big|,\nonumber\\
&=&O_{p}\Big({\big (\frac{\log \log n}{n}\big )}^{\frac{1}{2}}\Big) .
\end{eqnarray}
Now, start by writing
\begin{eqnarray}\label{fad24}
{\big(h_n {\hat{\sigma}}^{2}(p)\big)}^{-\frac{1}{2}}\bigg\{{\big(nh_n\big)}^{\frac{p}{2}} I_{n,n}(p)-\hat{m}(p)\bigg\}&=&{\big(h_n {\hat{\sigma}}^{2}(p)\big)}^{-\frac{1}{2}}\bigg\{{\big(nh_n\big)}^{\frac{p}{2}} I_{n,n}(p)-m(p)\bigg\}\nonumber\\
&\quad + &{\big(h_n {\hat{\sigma}}^{2}(p)\big)}^{-\frac{1}{2}}\Big(\hat{m}(p)-m(p)\Big)\nonumber\\
&:=& Z_n+V_n.
\end{eqnarray}
By Lemma \ref{fad21}
\begin{eqnarray*}
V_n=o(1)\quad a.s.
\end{eqnarray*}
Let  $\Gamma_{n,3}(t)$ be the term of defined in Lemma \ref{llle1.2}. Then
\begin{eqnarray}\label{fad25}
I_{n,n}(p)&=&\int_{0}^{T} {\bigg|h^{-1}_n \int_{\mathbb{R}} K\big(\frac{t-x}{h_n}\big)d\Big(F_{n,n}(x)-F_{n}(x)\Big)\bigg|}^{p}{\big(tv_n\big)}^{\frac{p}{2}}f_{n}(t) dt\nonumber\\
&=& n^{-\frac{p}{2}}h^{-p}_n\int_{0}^{T}\bigg({\bigg|\int_{\mathbb{R}} K\big(\frac{t-x}{h_n}\big)d\alpha_{n,n}(x)\bigg|}^{p}-{\Big|\Gamma_{n,3}(t)\Big|}^{p}\bigg){\big(tv_n\big)}^{\frac{p}{2}}f_n(t) dt\nonumber\\
&\quad +& n^{-\frac{p}{2}} h^{-p}_{n} \int_{0}^{T} {\Big|\Gamma_{n,3}(t)\Big|}^{p}{\big(tv_n\big)}^{\frac{p}{2}}f_{n}(t) dt\nonumber\\
&:=& R_n+S_n.
\end{eqnarray}
Using the inequality \eqref{setar}
% the inequality
%\begin{eqnarray*}
%\Big|{\big|a(x)\big|}^{p}-{\big|b(x)\big|}^{p}\Big|&\leq& p2^{p-1} {\big|a(x)-b(x)\big|}^{p} + p 2^{p-1} {\big|%b(x)\big|}^{p-1}\big|a(x)-b(x)\big|,
%\end{eqnarray*}
%(for $p\geq 1$),
 we can write
\begin{eqnarray}\label{fad26}
n^{\frac{p}{2}}h^{p}_{n}R_{n}&\leq& p2^{p-1} \bigg\{\int_{0}^{T}{\Big|\int_{\mathbb{R}}K\big(\frac{t-x}{h_n}\big)d\Big(\alpha_{n,n}(x)-\Gamma_{n,n}(x)\Big)\Big|}^{p}{\big(tv_n\big)}^{\frac{p}{2}}f_{n}(t)
dt\bigg\}\nonumber\\
&\quad +& p2^{p-1}{\bigg(\int_{0}^{T}{\Big|\Gamma_{n,3}(t)\Big|}^{p}{\big(tv_n\big)}^{\frac{p}{2}}f_n(t) dt\bigg)}^{\frac{(p-1)}{p}}\nonumber\\
&\quad\times&{\bigg(\int_{0}^{T}{\bigg|\int_{\mathbb{R}}K\big(\frac{t-x}{h_n}\big)d\Big(\alpha_{n,n}(x)-\Gamma_{n,n}(x)\Big)\bigg|}^{p}{\big(tv_n\big)}^{\frac{p}{2}}f_n(t)
dt\bigg)}^{\frac{1}{p}}.\nonumber\\
\quad
\end{eqnarray}
Next,  note that by Lemma \ref{eeeq*} and the bounded variation assumption on $K$
\begin{eqnarray*}
\bigg|\int_{\mathbb{R}}K\big(\frac{t-x}{h_n}\big)d\Big(\alpha_{n,n}(x)-\Gamma_{n,n}(x)\Big)\bigg|&\leq& \int_{\mathbb{R}}\Big|\alpha_{n,n}(t-yh_n)-\Gamma_{n,n}(t-yh_n)\Big|\Big|dK(y)\Big|\\
&\leq &\sup_{x} |\alpha_{n,n}(x)-\Gamma_{n,n}(x)|\int_{\mathbb{R}}\Big|dK(y)\Big|\\
&=&O_{p}\big(n^{-\frac{1}{r}}\big).
\end{eqnarray*}
Therefore,
\begin{eqnarray}\label{fad27}
RHS\quad of \eqref{fad26}=O_{p}\big(n^{-\frac{p}{r}}\big)+O_{p}\big(n^{-\frac{1}{r}}\big){\bigg(\int_{0}^{T}{\Big|\Gamma_{n,3}(t)\Big|}^{p}{\big(tv_n\big)}^{\frac{p}{2}}f_n(t)dt\bigg)}^{\frac{(p-1)}{p}}.\quad
\end{eqnarray}
%On the other hand, by the proofs of Lemma \ref{llle1.2} and Lemma \ref{le2}
On the other hand, by the proof of Lemma 2 of Fakoor and Zamini \cite{6}, one can write 
\begin{eqnarray*}
{\Gamma}_{n}^{(2)}(x)=P(x){\Gamma}_{n}^{(1)}(x)+o_{p}(h_n),
\end{eqnarray*}
where ${\Gamma}_{n}^{(2)}(x)=\int K\big(\frac{x-y}{h_n}\big) P(y)dW(y),$ with $P(x)={({\sigma}'(x))}^{1/2}$ and ${\Gamma}_{n}^{(1)}(x)=\int K\big(\frac{x-y}{h_n}\big)dW(y).$ But, $\big|{\Gamma}_{n}^{(1)}(x)\big|\stackrel{D}{=} |N|{(h_n)}^{1/2}{\big(\int {K(u)}^{2} du\big)}^{1/2},$ hence by $\bold F(1),$
\begin{eqnarray}\label{noghte}
{\sup}_{0<x<\tau}\big|{\Gamma}_{n}^{(2)}(x)\big|=o_{p}(h_{n}^{1/2}).
\end{eqnarray}
\eqref{noghte} and the proof of Lemma \ref{llle1.2} conclude that 
\begin{eqnarray}\label{fad28}
\sup_{0<x<\tau}\big|\Gamma_{n,3}(t)\big|=O_{p}\big(h^{\frac{1}{2}}_{n}\big).
\end{eqnarray}
Hence, by \eqref{fad22},\eqref{2} and \eqref{fad28}
\begin{eqnarray*}
&\bigg|&\int_{0}^{T}{\Big|\Gamma_{n,3}(t)\Big|}^{p}{\big(tv_n\big)}^{\frac{p}{2}}f_n(t)dt-\int_{0}^{T}{\Big|\Gamma_{n,3}(t)\Big|}^{p}{\big(tv\big)}^{\frac{p}{2}}f(t)dt\bigg|\\
&\leq& \int_{0}^{T}{\Big|\Gamma_{n,3}(t)\Big|}^{p}{\big(tv_n\big)}^{\frac{p}{2}}\Big|f_n(t)-f(t)\Big|dt\\
&\quad +&\Big|v^{\frac{p}{2}}_n-v^{\frac{p}{2}}\Big|\int_{0}^{T}{\Big|\Gamma_{n,3}(t)\Big|}^{p}{\big(t\big)}^{\frac{p}{2}}\big|f(t)\big|dt\\
&\leq&{\big(Tv_n\big)}^{\frac{p}{2}}\sup_{x}\Big|f_n(x)-f(x)\Big|\int_{0}^{T}{\Big|\Gamma_{n,3}(t)\Big|}^{p}dt\\
&\quad +&T^{\frac{p}{2}}\sup_{x}\big|f(x)\big|\Big|v^{\frac{p}{2}}_n-v^{\frac{p}{2}}\Big|\int_{0}^{T}{\Big|\Gamma_{n,3}(t)\Big|}^{p}dt\\
&=&O_{p}\Big({(h_{n})}^{p/2}\Big)\bigg(h^{-1}_{n}O_{p}{\Big(\frac{\log\log n}{n}\Big)}^{\frac{1}{2}}+O_{p}(h_n)\bigg).
\end{eqnarray*}
This last result and Lemma \ref{llle1.2} imply  that
\begin{eqnarray}\label{fad29}
{\big(h^{p+1}_{n}{\sigma}^{2}(p)\big)}^{-\frac{1}{2}}\bigg\{\int_{0}^{T}{\Big|\Gamma_{n,3}(t)\Big|}^{p}{\big(tv_n\big)}^{\frac{p}{2}}f_n(t) dt-h^{\frac{p}{2}}_{n}m(p)\bigg\}\stackrel{D}{\longrightarrow} N(0,1).\quad
\end{eqnarray}
Consequently
\begin{eqnarray}\label{fad30}
\int_{0}^{T}{\Big|\Gamma_{n,3}(t)\Big|}^{p}{\big(tv_n\big)}^{\frac{p}{2}}f_n(t) dt=O_{p}\big(h^{\frac{p}{2}}_{n}\big).
\end{eqnarray}
Putting together \eqref{fad26},\eqref{fad27} and \eqref{fad30}, one finds
\begin{eqnarray}\label{fad31}
R_n=n^{-\frac{p}{2}}h^{-p}_{n}O_{p}\big(n^{-\frac{p}{r}}\big)+n^{-\frac{p}{2}}h^{-p}_{n}O_{p}\big(n^{-\frac{1}{r}}\big)O_p(h^{\frac{(p-1)}{2}}_{n}).
\end{eqnarray}
The term $Z_n$ that appears in \eqref{fad24} can now be handled as follows
\begin{eqnarray*}
Z_n&=&{\big(h_n{\hat{\sigma}}^{2}(p)\big)}^{-\frac{1}{2}}\bigg\{{\big(nh_n\big)}^{\frac{p}{2}}\Big(R_n+S_n\Big)-m(p)\bigg\}\\
&=&{\big(h_n{\hat{\sigma}}^{2}(p)\big)}^{-\frac{1}{2}}{\big(nh_n\big)}^{\frac{p}{2}}R_n +{\big(h_n{\hat{\sigma}}^{2}(p)\big)}^{-\frac{1}{2}}\bigg\{{\big(nh_n\big)}^{\frac{p}{2}}S_n-m(p)\bigg\}\\
&:=& Z_{n,1}+Z_{n,2}.
\end{eqnarray*}
The fact that $\frac{{\hat{\sigma}}^{2}(p)}{{\sigma}^{2}(p)}=1+o(1)\quad a.s.,$ together with \eqref{fad29} imply
\begin{eqnarray*}
Z_{n,2}={\big(h^{p+1}_{n}{\hat{\sigma}}^{2}(p)\big)}^{-\frac{1}{2}}\bigg\{\int_{0}^{T}{\Big|\Gamma_{n,3}(t)\Big|}^{p}{\big(tv_n\big)}^{\frac{p}{2}}f_n(t) dt-h^{\frac{p}{2}}_{n}m(p)\bigg\}\stackrel{D}{\longrightarrow} N(0,1).\quad
\end{eqnarray*}
As for $Z_{n,1},$ the bound in \eqref{fad31} gives
\begin{eqnarray*}
\big|Z_{n,1}\big|&=& h^{-\frac{(p+1)}{2}}_{n}n^{-\frac{p}{r}}o_p(1)+h^{-1}_nn^{-\frac{1}{r}}o_p(1)\\
&\leq& h^{-p}_nn^{-\frac{p}{r}}o_p(1)+h^{-1}_nn^{-\frac{1}{r}}o_p(1)\\
&=&o_p(1).
\end{eqnarray*}
This completes the proof of Theorem \ref{fad23}.
\end{proof}
\section*{Appendix}
In order to make the proof of the main result easier, some auxiliary results and notations are needed.

The following inequality will be useful later on. Let $1\leq p<\infty,$ then for functions $q$ and $u$ in $L_p$ we have
\begin{eqnarray}\label{setar}
\int_{0}^{\infty}\Big|{\big|q(t)\big|}^{p}-{\big|u(t)\big|}^{p}\Big| d(\mu(t))&\leq& p2^{p-1}\int \big|q(t)-u(t){\big|}^{p}d(\mu(t))\nonumber\\
&\quad +&p2^{p-1}{\big(\int_{0}^{\infty}{|u(t)|}^{p}d(\mu(t))\big)}^{1-\frac{1}{p}}\nonumber\\
&\quad \times&{\big(\int_{0}^{\infty}{\big|q(t)-u(t)\big|}^{p}d\mu(t)\big)}^{1/p}.
\end{eqnarray}
\begin{lemma}\label{eeeq*}
Under Assumption $\bold{G(1)},$ one can write
\[\sup_{0<t<\tau}\Big|\alpha_{n,n}(t)-\Gamma_{n,n}(t)\Big|=O(n^{-1/r}) \quad a.s.\]
\end{lemma} 
\begin{proof}
Applying  \eqref{eeeq1} and \eqref{eeeq2}, one can obtain 
\begin{eqnarray*}
\sup_{0<t<\tau}\Big | \alpha_{n,n}(t)-\Gamma _{n,n}(t) \Big |
% &\leqslant& v_n^{-1}\sup_{0<t<\tau}\bigg | %%\int_{0}^t y^{-1}d \Big(\beta_{n,n}(y)-B_{n,n}(G(y)\Big) \bigg | \\
%&\quad +&\sup_{0<t<\tau}\frac{\Big(\int_{0}^ty^{-1}dG_{n,n}(y)\Big)}{v_{nn} v_n}\bigg | \int_{0}^{\tau}%y^{-1}d\Big(B_{n,n}(G(y))-\beta_{n,n}(y)\Big)\bigg |\\
%&\leq&v^{-1}_n \sup_{0<t<\tau}\bigg|\frac{\beta_{n,n}(t)-B_{n,n}\big(G(t)\big)}{t}\bigg|\\
%&\quad+&v^{-1}_n\int_{0}^{t} y^{-2} \Big|\beta_{n,n}(y)-B_{n,n}\big(G(y)\big)\Big|dy\\
%&\quad +&v^{-1}_n \sup_{0<t<\tau}\bigg|\frac{\beta_{n,n}(t)-B_{n,n}\big(G(t)\big)}{t}\bigg|\\
%&\quad+&v^{-1}_n\int_{0}^{\tau} y^{-2} \Big|\beta_{n,n}(y)-B_{n,n}\big(G(y)\big)\Big|dy\\
&\leq&
2v^{-1}_n \sup_{0<t<\tau}\bigg|\frac{\beta_{n,n}(t)-B_{n,n}\big(G(t)\big)}{t}\bigg|\\
&\quad+&2v^{-1}_n\int_{0}^{\tau} y^{-2} \Big|\beta_{n,n}(y)-B_{n,n}\big(G(y)\big)\Big|dy\\
%&\quad+&2v^{-1}_n\tau\sup_{0<t<\tau}\Big(\frac{{(G(t))}^{1/r}}{t^2}\Big)\sup_{0<t<\tau}\bigg|\frac{\beta_{n,n}(t)-B_{n,n}\big(G(t)\big)}{{\big(G(t)(1-G(t))\big)}^{1/2-1/r}}\bigg|\\
%&\quad+&2v^{-1}_n\sup_{0<t<\tau}\bigg|\frac{\beta_{n,n}(t)-B_{n,n}\big(G(t)\big)}{{\big(G(t)(1-G(t))\big)}^{1/2-1/r}}\bigg|\\
%&\quad \times&\int_{0}^{\tau}y^{-2}{\big(G(y)(1-G(y))\big)}^{1/2-1/r}\\
&\leq&
2v^{-1}_n\tau\sup_{0<t<\tau}\Big(\frac{{(G(t))}^{1/r}}{t^2}\Big)\sup_{0<t<\tau}\bigg|\frac{\beta_{n,n}(t)-B_{n,n}\big(G(t)\big)}{{\big(G(t)(1-G(t))\big)}^{1/2-1/r}}\bigg|\\
&\quad+&2v^{-1}_n\sup_{0<t<\tau}\bigg|\frac{\beta_{n,n}(t)-B_{n,n}\big(G(t)\big)}{{\big(G(t)(1-G(t))\big)}^{1/2-1/r}}\bigg|\\
&\quad \times&\int_{0}^{\tau}y^{-2}{\big(G(y)\big)}^{1/r}\\
&=&O_{p}(n^{-1/r})+O_{p}(n^{-1/r})\\
&=&O_{p}(n^{-1/r}),
\end{eqnarray*}
where last equality results from Assumption $\bold {G(1)}$ 
 and Remark 2 of Cs\"{o}rg\H{o} and Mason \cite{3}.
\end{proof} 
\begin{lemma}\label{llle1.2}
Let $\Gamma_{n,3}(t)=\int_{\mathbb{R}}K\big(\frac{t-x}{h_n}\big) d \Gamma_{n,n}(x)$.
Suppose that  Assumptions $\bold{K(1)}$-$\bold{K(3)},$  $\bold{C(1)},$ $\bold{G(1)},$ $\bold{F(1)}$ and conditions
\begin{eqnarray*}
h_{n}\rightarrow 0,\quad \frac{\log n}{nh^{2}_{n}}\rightarrow0,
\end{eqnarray*}
hold, then we can write
\begin{eqnarray*}
(h_n^{p+1}\sigma^2(p))^{-\frac{1}{2}}
\bigg\{\int_{0}^T \big |\Gamma_{n,3}(t) \big |^p (tv)^{\frac{p}{2}}f(t)dt-h_n^{\frac{p}{2}}m(p)\bigg\}\stackrel{D}{\longrightarrow}N(0,1).
\end{eqnarray*}
\end{lemma}
\begin{proof}At first, note that 
\begin{eqnarray}\label{eeeqI'}
\Gamma_{n,3}(t)&=&\frac{1}{v_n}\int_{\mathbb{R}}K\big(\frac{t-x}{h_n}\big)x^{-1}dB_{n,n}(G(x))\nonumber\\
&\quad -&\frac{\Big(\int_{0}^{\tau}y^{-1}dB_{n,n}(G(y))\Big)}{v_n v_{nn}} \int_{\mathbb{R}}K\big(\frac{t-x}{h_n}\big)x^{-1}dG_{n,n}(x)\nonumber\\
&:=&B_n^{(1)}(t)+B_n^{2}(t).
\end{eqnarray}
The fact that $\{B_{n,n}(z), 0\leqslant z\leqslant 1\} \stackrel{D}{=}\{W(z)-zW(1), 0\leqslant z \leqslant 1\}$
for all $n\geqslant 1$, implies that 
\begin{eqnarray}\label{eeeq1.8}
B_n^{(1)}(t)&\stackrel{D}{=}&\frac{1}{v_n}\int_{\mathbb{R}}K\big(\frac{t-x}{h_n}\big)x^{-1}d\Big(W\big(G(x)\big)-G(x)W(1)\Big)\nonumber\\
&:=&B_n^{\prime(1)}(t).
\end{eqnarray}
Also, with using Theorem of James \cite{10} and \textbf{G(1)},  one can obtain,  for any $0<\delta <\frac{1}{2}-\frac{1}{r}$
\begin{eqnarray}\label{do}
 \big |v-v_{n}\big |&=&\Big |\int_{0}^{\tau} y^{-1} d\big (G_{n}(y)-G(y)\big )\Big |\nonumber \\
%&=& n^{-\frac{1}{2}} \Big |\int_{0}^{\tau} y^{-1} d{\beta}_{n}(y)\Big |
&=& n^{-\frac{1}{2}}\big |\int_{0}^{\tau} y^{-2} {\beta}_{n}(y)dy\big |\nonumber\\
&\leq &   n^{-\frac{1}{2}}\sup_{0<y<\tau}{(G(y))}^{\delta-\frac{1}{2}}\big |\beta_{n}(y) \big  |\Big |\int_{0}^{\tau} y^{-2} {(G(y))}^{\frac{1}{2}-\delta}dy\Big |\nonumber\\
&=& O\Big ({\big(\frac{\log \log n}{n}\big)}^{\frac{1}{2}}\Big )\quad a.s.
\end{eqnarray}
%Set
%\begin{eqnarray}\label{eeeq1.9}
%B_n^{\prime(1)}(t):=H_n^{(1)}(t)+H_n^{(2)}(t),
%\end{eqnarray}
%where
%\begin{eqnarray*}
%H_n^{(1)}(t):=\frac{1}{v_n}\int_{\mathbb{R}}K\big(\frac{t-x}{h_n}\big)x^{-1}d\Big(W(G(x))\Big),
%\end{eqnarray*}
%and
%\begin{eqnarray*}
%H_n^{(2)}(t):=-\frac{W(1)}{v_n}\int_{\mathbb{R}}K\big(\frac{t-x}{h_n}\big)x^{-1}d\big(W(G(x))\big).
%\end{eqnarray*}
%But
%\begin{eqnarray*}
%H_n^{(1)}(t)\stackrel{D}{=}\frac{v}{v_n}\Gamma_n^{(2)}(x)
%\end{eqnarray*}
 \eqref{setar}, \eqref{do} and Lemma 2 of Fakoor and Zamini \cite{6}, imply that
%\begin{eqnarray}\label{eeeq1.10}
%&\bigg|&\int_{0}^{T}\Big |H_n^{(1)}(t)\Big |^p(tv)^{\frac{p}{2}}f(t)dt-\int_{0}^{T}\Big |\Gamma_n^{(2)}(t)\Big |^p(tv)^{\frac{p}{2}}f(t)dt\bigg|\nonumber\\
%&\leqslant& p2^{p-1}\Big|\frac{v-v_n}{v_n}\Big|^p \int_{0}^{T}\Big |\Gamma_n^{(2)}(t) \Big|^p(tv)^{\frac{p}{2}}f(t)dt\nonumber\\
%&\quad +&\Big |\frac{v-v_n}{v_n} \Big |p2^{p-1}\bigg(\int_{0}^T\Big |\Gamma _n^{(2)}(t)\Big|^p(tv)^{\frac{p}{2}}f(t)dt\bigg)^{1-\frac{1}{p}}\nonumber\\
%&\quad \times& \bigg(\int_{0}^T|\Gamma _n^{(2)}(t)|^p(tv)^{\frac{p}{2}}f(t)dt\bigg)^{\frac{1}{p}}\nonumber\\
%&=&O_p\Big(\big(\frac{\log \log n}{n}\big)^{\frac{p}{2}}\Big)O_p(h_n^{\frac{p}{2}})
%\end{eqnarray}
%\eqref{eeeq1.10} and lemma \ref{le2} imply that
\begin{eqnarray*}
\big(h_n^{p+1}\sigma^2(p)\big)^{-\frac{1}{2}}\bigg\{ \int_{0}^T\Big |\frac{v}{v_n}\Gamma_n^{(2)}(t) \Big |^p(tv)^{\frac{p}{2}}f(t)dt-h_n^{\frac{p}{2}}m(p)\bigg\} \stackrel{D}{\longrightarrow}N(0,1),\quad
\end{eqnarray*}
where $\Gamma_n^{(2)}(x)$ is introduced in Lemma 2 of Fakoor and Zamini \cite{6}.

Now, since
\begin{eqnarray*}
H_n^{(1)}(t):=\frac{1}{v_n}\int_{\mathbb{R}}K\big(\frac{t-x}{h_n}\big)x^{-1}d\Big(W(G(x))\Big)\stackrel{D}{=}\frac{v}{v_n}\Gamma_n^{(2)}(t),
\end{eqnarray*}
one can write
\begin{eqnarray}\label{eeeq1.11}
\big(h_n^{p+1}\sigma^2(p)\big)^{-\frac{1}{2}}\bigg\{ \int_{0}^T\Big | H_n^{(1)}(t)\Big |^p(tv)^{\frac{p}{2}}f(t)dt-h_n^{\frac{p}{2}}m(p)\bigg\} \stackrel{D}{\longrightarrow}N(0,1).\quad
\end{eqnarray}
Also
\begin{eqnarray*}
H_n^{(2)}(t):=-W(1)\frac{vh_n}{v_n}\int_{\mathbb{R}} K(u)f(t-uh_n)du,
\end{eqnarray*}
is normally distributed with mean $0$ and variance
\begin{eqnarray*}
\sum_{n,t}=\frac{v^2}{v_n^2}h_n^2\bigg(\int_{\mathbb{R}} K(u)f(t-uh_n)du\bigg)^2.
\end{eqnarray*}
Therefore, for each $t$,
\begin{eqnarray*}
\Big|H_n^{(2)}(t) \Big|\stackrel{D}{=}|N|\frac{v}{v_n}h_n\bigg|\int_{\mathbb{R}} K(u)f(t-uh_n)du\bigg|.
\end{eqnarray*}
Since
\begin{eqnarray*}
\frac{v}{v_n} h_n \bigg |\int_{\mathbb{R}} K(u) f(t-uh_n) du \Big| \leqslant \frac{v}{v_n} h_n  \sup_{0<x<\tau}f(x) \int_{-1}^{+1} |K(u)| du =O_p (h_n),
\end{eqnarray*}
hence
\begin{eqnarray}
\label{eeeq**}
H^{(2)} _n(t) = O_p (h_n).
\end{eqnarray}
Now, by   \eqref{setar}, \eqref{eeeq1.11} and  \eqref{eeeq**}, one can see
\begin{eqnarray*}
\bigg| \int ^T_{0} \Big| {B'}_n^{(1)} (t)  \Big|^p (tv) ^{\frac{p}{2}} f(t) dt &-&\int ^T _{0} \Big| H_n ^{(1)} (t) \Big| ^p  (tv)^{\frac{p}{2}} f(t) dt \bigg|
\\
&\leqslant& p2^{p-1} \int ^T_0 \Big| H^{(2)} _n (t) \Big| ^p (tv)^{\frac{p}{2}} f(t) dt \\
&\quad+& p2^{p-1} \bigg( \int ^T_0 \Big|  H^{(1)} _n (t)\Big|^p (tv)^{\frac{p}{2}}  f(t) dt\bigg)^{1-\frac{1}{p}} \\
&\quad \times & \bigg(\int ^T_0 \Big| H_n ^{(2)} (t)  \Big|^p (tv) ^{\frac{p}{2}} f(t) dt \bigg) ^{\frac{1}{p}}\\
&=& O_p (h_n ^{\frac{p+1}{2}})
\end{eqnarray*}
Last result with together  \eqref{eeeq1.11}  and \eqref{eeeq1.8} conclude that
\begin{eqnarray}
\label{eeeqII'}
\Big(h_n ^{p+1} \sigma ^2 (p)\Big)^{-\frac{1}{2}} \left\{\int ^T_0 \Big| B^{(1)}_n (t) \Big|^p (tv)^{\frac{p}{2}} f(t) dt - h_n ^{\frac{p}{2}} m(p) \right\} \stackrel{D}{\longrightarrow} N(0,1).
\quad
\end{eqnarray}

The term $B_n ^{(2)} (t)$  can be handled as follows. \\
Since $\{B_{n,n} (z) , 0\leqslant z \leqslant 1 \} \stackrel{D}{=} \{ W(z) - zW(1) , 0 \leqslant z \leqslant 1\}$ for all $n \geqslant 1$, hence
\begin{eqnarray}\label{fad1}
\Big|B_n ^{(2)}(t)\Big|\stackrel{D}{=}\Big|\frac{1}{v_n} \int ^{\tau} _0 y^{-1} d \Big[W(G(y)) - G(y) W(1) \Big]\frac{1}{v_{nn}}\int_{\mathbb{R}}K(\frac{t-x}{h_n}) x^{-1} dG_{n,n}(x)\Big|.
\quad
\end{eqnarray}
But
\begin{eqnarray}\label{fad2}
&\Big|&\frac{1}{v_n} \int ^{\tau} _0 y^{-1} d \Big(W(G(y)) - G(y) W(1) \Big)\Big|
\Big|\frac{1}{v_{nn}}\int_{\mathbb{R}}K(\frac{t-x}{h_n}) x^{-1} dG_{n,n}(x)\Big|\nonumber\\
&\leq&\bigg(\frac{1}{v_n}\bigg|\int_{0}^{\tau}y^{-1}d\Big(W\big(G(y)\big)\Big)\bigg|
+\frac{1}{v_n}\bigg|W(1)\int_{0}^{\tau}y^{-1} g(y) dy\bigg|\bigg)\nonumber\\
&\quad\times&\bigg|\frac{1}{v_{nn}}\int_{\mathbb{R}}K(\frac{t-x}{h_n}) x^{-1} dG_{n,n}(x)\bigg|\nonumber\\
&:=& \frac{1}{v_n}\Big(A_1+A_2\Big)\times A_3.
\end{eqnarray}
Now, since
\begin{eqnarray*}
\big|A_1\big|\stackrel{D}{=}|N|{\bigg(\int_{0}^{\tau}y^{-2}g(y)dy\bigg)}^{\frac{1}{2}},
\end{eqnarray*}
therefore,
\begin{eqnarray}\label{fad3}
\frac{1}{v_n}A_1=O_{p}(1).
\end{eqnarray}
Also
\begin{eqnarray*}
\big|A_2\big|\stackrel{D}{=}|N|{\bigg(\int_{0}^{\tau}y^{-1}g(y)dy\bigg)},
\end{eqnarray*}
hence
\begin{eqnarray}\label{fad4}
\frac{1}{v_n}A_2=O_{p}(1).
\end{eqnarray}
Next, let $\big\{B_{n,n}(x); 0\leq x \leq 1\big\}$ be the sequence of Brownian brideges in \eqref{eeeqII}. Now, one can write
\begin{eqnarray}\label{variance}
A_3&=&\frac{1}{v_{n,n}}\bigg|\int_{\mathbb{R}}x^{-1}K\big(\frac{t-x}{h_n}\big)dG_{n,n}(x)\bigg|\nonumber\\
& \leq & \tau \bigg|\int_{\mathbb{R}}x^{-1}K\big(\frac{t-x}{h_n}\big)dG_{n,n}(x)\bigg|\nonumber\\
& \leq & n^{-1/2}\bigg|\int_{\mathbb{R}}x^{-1}K\big(\frac{t-x}{h_n}\big)d\Big(\beta_{n,n}(x)-B_{n,n}(G(x))\Big)\bigg|\nonumber\\
&\quad +& n^{-1/2}\bigg|\int_{\mathbb{R}}x^{-1}K\big(\frac{t-x}{h_n}\big)d\Big(B_{n,n}(G(x))\Big)\bigg|\nonumber\\
&\quad +& n^{-1/2}\bigg|\int_{\mathbb{R}}x^{-1}K\big(\frac{t-x}{h_n}\big)d\Big(\beta_{n}(x)\Big)\bigg|\nonumber\\
&\quad +& \bigg|\int_{\mathbb{R}}x^{-1}K\big(\frac{t-x}{h_n}\big)d\Big(G(x)\Big)\bigg|\nonumber\\
&:=& L^{(1)}_{n}(t)+ L^{(2)}_{n}(t)+ L^{(3)}_{n}(t)+ L^{(4)}_{n}(t).
\end{eqnarray}
 \eqref{eeeqII} implies that
\begin{eqnarray}\label{fad5}
L^{(1)}_{n}(t)&=&n^{-\frac{1}{2}}\bigg|\int_{\mathbb{R}}x^{-1} K\big(\frac{t-x}{h_n}\big)d\Big(\beta_{n,n}(x)-B_{n,n}\big(G(x)\big)\Big)\bigg| \nonumber\\
&=&n^{-\frac{1}{2}}\bigg|\int_{\mathbb{R}}\Big(\beta_{n,n}(t-uh_n)-B_{n,n}\big(G(t-uh_n)\big)\Big)d\Big(\frac{K(u)}{t-uh_n}\Big)\bigg|\nonumber\\
&\leq& n^{-\frac{1}{2}}\sup_{x>0} \Big|\beta_{n,n}(x)-B_{n,n}\big(G(x)\big)\Big|\int_{\mathbb{R}}\bigg|d\Big(\frac{K(u)}{t-uh_n}\Big)\bigg|\nonumber\\
&=&O\Big(n^{-1}\log n\Big)\bigg(\int_{-1}^{+1}\frac{K^{'}(u)}{t-uh_n}+\int_{-1}^{+1}\frac{h_n K(u)}{{(t-uh_n)}^{2}} \bigg)\nonumber\\
&\leq& \max{\Big(\sup_{x}\big|K(x)\big|,\sup_{x}\big|K^{'}(x)\big|\Big)}O\Big(n^{-1}\log n\Big)\nonumber\\
&\quad\times&\bigg(\frac{1}{h_n}\log \Big(\frac{t+h_n}{t-h_n}\Big)+\frac{2h_n}{t^2-h^{2}_{n}}\bigg)\nonumber\\
&=& O\Big(\frac{\log n}{nh_n}\Big)\Big(O(1)+O(1)\Big)\nonumber\\
&=&O\Big(\frac{\log n}{nh_n}\Big).
\end{eqnarray}
To deal with $L^{2}_{n}(t),$ observe that 
\begin{eqnarray}\label{fad6}
L^{2}_{n}(t)&\stackrel{D}{=}&n^{-\frac{1}{2}}\bigg|\int_{\mathbb{R}}x^{-1}K\big(\frac{t-x}{h_n}\big)d\Big(W\big(G(x)\big)-G(x)W(1)\Big)\bigg|\nonumber\\
&\leq& n^{-\frac{1}{2}}\bigg|\int_{\mathbb{R}}x^{-1}K\big(\frac{t-x}{h_n}\big)dW\big(G(x)\big)\bigg|\nonumber\\
&\quad +&n^{-\frac{1}{2}}\bigg|\int_{\mathbb{R}}x^{-1} K\Big(\frac{t-x}{h_n}\Big)W(1) g(x) dx\bigg|\nonumber\\
&:=& A^{'}_{1}(t)+A^{'}_{2}(t).
\end{eqnarray}
Using the facts that
\begin{eqnarray*}
\Big|h^{-\frac{1}{2}}_{n}n^{\frac{1}{2}}A^{'}_{1}(t)\Big|&\stackrel{D}{=}&h^{-\frac{1}{2}}_{n}\bigg|\int_{\mathbb{R}}x^{-1} K\Big(\frac{t-x}{h_n}\Big)g^{\frac{1}{2}}(x)d W(x)\bigg|.\\
&\stackrel{D}{=}&h^{-\frac{1}{2}}_{n}\big|N\big|{\bigg(\int_{\mathbb{R}}x^{-2}K^{2}\Big(\frac{t-x}{h_n}\Big)g(x) dx\bigg)}^{\frac{1}{2}},
\end{eqnarray*}
and
\begin{eqnarray*}
h^{-\frac{1}{2}}_{n}{\bigg(\int_{\mathbb{R}}x^{-2}K^{2}\Big(\frac{t-x}{h_n}\Big)g(x) dx\bigg)}^{\frac{1}{2}}&=&{\bigg(\int_{\mathbb{R}}v\frac{f(t-uh_n)}{(t-uh_n)}K^{2}(u) du\bigg)}^{\frac{1}{2}}\\
&\longrightarrow&{(\frac{f(t)}{t})}^{\frac{1}{2}}{\bigg(v\int_{-1}^{+1} K^{2}(u)du\bigg)}^{\frac{1}{2}},
\end{eqnarray*}
it is not difficult to show that 
\begin{eqnarray}\label{fad7}
A^{'}_{1}(t)=O_{p}\Big(n^{-\frac{1}{2}}h^{\frac{1}{2}}_{n}\Big)\quad a.s.
\end{eqnarray}
Also
\begin{eqnarray*}
\big|h^{-1}_{n} n^{\frac{1}{2}} A^{'}_{2}(t)\big|&\stackrel{D}{=}&\big|N\big|v \bigg|\int_{\mathbb{R}}K(u)f(t-uh_n) du\bigg|,\nonumber\\
&\longrightarrow& \big|N\big|v f(t) \int_{\mathbb{R}}K(u) du\nonumber\\
&=&O_{p}(1),
\end{eqnarray*}
therefore, it follows that
\begin{eqnarray}\label{fad8}
A^{'}_{2}(t)=O_{p}(h_n n^{-\frac{1}{2}}).
\end{eqnarray}
\eqref{fad6},\eqref{fad7} and \eqref{fad8} imply that
\begin{eqnarray}\label{fad9}
L^{(2)}_{n}(t)=O_{p}\big(n^{-\frac{1}{2}} h^{\frac{1}{2}}_{n}\big).
\end{eqnarray}

From Koml\'{o}s et al. \cite{12}, there  exists a sequence of Brownian bridges $\{B_n (t) , 0\leqslant t \leqslant 1\}$ such that
\begin{eqnarray}
\label{eeeq100}
\sup_{-\infty <x<\infty} \bigg| \beta_n (x) - B_n \big( G(x) \big) \bigg| = O_p \Big( n^{-\frac{1}{2}} \log n \Big).
\end{eqnarray}
Consequently,
\begin{eqnarray}\label{fad10}
L^{(3)}_{n}(t)&=&\bigg|n^{-\frac{1}{2}}\int_{\mathbb{R}}x^{-1} K\big(\frac{t-x}{h_n}\big)d\beta_{n}(x)\bigg|\nonumber\\
&\leq& \bigg|n^{-\frac{1}{2}}\int_{\mathbb{R}}x^{-1} K\big(\frac{t-x}{h_n}\big)d\Big(\beta_{n}(x)-B_{n}\big(G(x)\big)\Big)\bigg|\nonumber\\
&\quad +&\bigg|n^{-\frac{1}{2}}\int_{\mathbb{R}}x^{-1} K\big(\frac{t-x}{h_n}\big)d\Big(B_{n}\big(G(x)\big)\Big)\bigg|.\nonumber\\
\end{eqnarray}
With using \eqref{eeeq100} and similar to the term $L^{(1)}_{n}(t),$ one gets
\begin{eqnarray}\label{fad11}
&\bigg|&n^{-\frac{1}{2}}\int_{\mathbb{R}}x^{-1} K\big(\frac{t-x}{h_n}\big)d\Big(\beta_{n}(x)-B_{n}\big(G(x)\big)\Big)\bigg|\nonumber\\
&=&O_{p}\Big(\frac{\log n}{nh_{n}}\Big).
\end{eqnarray}
Similar to $L^{(2)}_{n}(t)$
\begin{eqnarray}\label{fad12}
&\bigg|&n^{-\frac{1}{2}}\int_{\mathbb{R}}x^{-1} K\big(\frac{t-x}{h_n}\big)d\Big(B_{n}\big(G(x)\big)\Big)\bigg|\nonumber\\
&=&O_{p}\Big(n^{-\frac{1}{2}} h^{\frac{1}{2}}_{n}\Big),
\end{eqnarray}
\eqref{fad10}-\eqref{fad12} conclude that 
\begin{eqnarray}\label{fad13}
L^{(3)}_{n}(t)= O_{p}\Big(\frac{\log n}{nh_n}\Big)+O_{p}\Big(n^{-\frac{1}{2}} h^{\frac{1}{2}}\Big).
\end{eqnarray}
To deal with the term of ${L}_{n}^{(4)}(t),$ observe that
\begin{eqnarray*}
h^{-1}_{n} L^{(4)}_{n}(t)&=&{h_n}^{-1}\bigg|\int_{\mathbb{R}}x^{-1} K\big(\frac{t-x}{h_n}\big)g(x)  dx\bigg|=v\bigg|\int_{\mathbb{R}} K(u) f(t-uh_n) du\bigg|\\
&\leq& v \sup_{0<x<\tau} f(x) \int_{-1}^{+1}K(u) du=O_{p}(1).
\end{eqnarray*}
Hence
\begin{eqnarray}\label{fad14}
L^{(4)}_{n}(t)=O_{p}(h_n).
\end{eqnarray}
Putting together \eqref{fad1}-\eqref{fad5}, \eqref{fad9} and  \eqref{fad13}-  \eqref{fad14} one concludes that
\begin{eqnarray}\label{fad15}
B_n ^{(2)} (t) =O_p \big(h_n\big ).
\end{eqnarray}
\eqref{fad15} implies that
\begin{eqnarray}
\label{eeeIII'}
\int ^T _0 \big| B_n ^{(2)} (t) \big| ^p (tv) ^{\frac{p}{2}} f(t) dt =O_p \big({h}^p_n
\big).
\end{eqnarray}
\eqref{setar}, \eqref{eeeqI'}, \eqref{eeeqII'} and \eqref{eeeIII'} follow
\begin{eqnarray*}
& \bigg|& \int ^T_0 |\Gamma _{n,3} (t)|^p (tv) ^{\frac{p}{2}} f(t) dt
-\int ^T_0 |B_n ^{(1)} (t)|^p (tv) ^{\frac{p}{2}}  f(t) dt \bigg| \\
& \leqslant& p 2^{p-1} \int ^T _0  |B_n ^{(2)} (t)|^p  (tv) ^{\frac{p}{2}} f(t) dt \\
&\quad+& p 2^{p-1} \bigg( \int ^T_0\big| B_n ^{(1)} (t) \big|^p  (tv)^{\frac{p}{2}} f(t) dt \bigg)^{1-\frac{1}{p}}  \bigg( \int ^T _0 \big|B_n ^{(2)} (t) \big| ^p (tv) ^{\frac{p}{2}} f(t) dt\bigg) ^{\frac{1}{p}} \\
&=&  O_p \big({h_n}^p
\big)\\
&\quad+&O_p (h_n ^{\frac{p-1}{2}}) O_p  \big(h_n\big).
\end{eqnarray*}
This last result and \eqref{eeeqII'} follow
\begin{eqnarray*}
(h_n ^{p+1} \sigma ^2 (p)) ^{-\frac{1}{2}}  \bigg\{ \int ^T _0\big| \Gamma _{n,3} (t) \big| ^p (tv)^{\frac{p}{2}} f(t) dt - h_n ^{\frac{p}{2}} m(p)\bigg\} \stackrel{D}{\longrightarrow} N(0,1).
\end{eqnarray*}
\end{proof}  
\begin{lemma}\label{fad21}
Suppose $f^{'}(x)$ exists and is bounded $(a.s.)$ on the $(0,\tau).$ Then under $\bold{K(1)}$, $\bold{K(3)}$, $\bold{K(4)}$, $\bold{G(1)}$ and conditions
\begin{eqnarray*}
h_n\rightarrow 0,\quad \frac{\log\log n}{nh^{3}_n}\rightarrow0,
\end{eqnarray*}
one can write\\

$h^{-\frac{1}{2}}_{n}\big|\hat{m}(p)-m(p)\big|\rightarrow0\quad a.s.$\quad and \quad ${\hat{\sigma}}^{2}(p)\rightarrow{\sigma}^{2}(p)\quad a.s.$
\end{lemma}
\begin{proof}
By \eqref{setar}, one can see
\begin{eqnarray*}
\big|\hat{m}(p)-m(p)\big|&\leq& C(K) \int_{0}^{T}\Big|f^{\frac{p+2}{2}}_{n}(x)-f^{\frac{p+2}{2}}(x)\Big|dx\\
&\leq& C(K) (\frac{p+2}{2}) 2^{(\frac{p+2}{2})-1}\int_{0}^{T}{\Big|f_{n}(x)-f(x)\Big|}^{(\frac{p+2}{2})} dx\\
&\quad +& C(K) \big(\frac{p+2}{2}\big)2^{(\frac{p+2}{2})-1}{\bigg(\int_{0}^{T} f^{\frac{(p+2)}{2}}(x) dx\bigg)}^{1-\frac{2}{(p+2)}}\\
&\quad \times& {\bigg(\int_{0}^{T} {\Big|f_n(x)-f(x)\Big|}^{\frac{(p+2)}{2}}(x) dx\bigg)}^{\frac{2}{(p+2)}},
\end{eqnarray*}
where $C(K)$ is a constant with respect to $K.$

Now, in order to show $h^{-\frac{1}{2}}_{n}\big|\hat{m}(p)-m(p)\big|\rightarrow 0\quad a.s.,$ it is sufficient to show that
\begin{eqnarray*}
h^{-\frac{1}{2}}_{n}{\bigg(\int_{0}^{T} {\Big|f_n(x)-f(x)\Big|}^{\frac{(p+2)}{2}}(x) dx\bigg)}^{\frac{2}{(p+2)}}\longrightarrow 0\quad a.s.
\end{eqnarray*}
 But by Lemma 5 of Fakoor and Zamini \cite{6}, and Assumptions $\bold{K(1)},$ $\bold{K(3)},$$\bold{K(4)},$ and $\bold{G(1)},$ one can get
\begin{eqnarray}\label{fad22}
\sup_{x}\big|f_{n}(x)-f(x)\big|&\leq& \sup_{x} \bigg|h^{-1}_{n} \int_{\mathbb{R}}K\big(\frac{x-y}{h_n}\big) dF_{n}(y)-h^{-1}_{n} \int_{\mathbb{R}}K\big(\frac{x-y}{h_n}\big)dF(y)\bigg|\nonumber\\
&\quad +&\sup_{x}\bigg|\int_{\mathbb{R}}\Big(f(x-h_nt)-f(x)\Big)K(t) dt\bigg|\nonumber\\
&\leq& h^{-1}_{n}\sup_{0<x<\tau}\big|F_{n}(x)-F(x)\big|\int_{\mathbb{R}}\big|dK(t)\big|\nonumber\\
&\quad +& h_{n} \sup_{x} \big|f^{'}(x)\big|\int_{-1}^{+1}|t| \big|K(t)\big| dt\nonumber\\
&=&h^{-1}_{n}O\Big({\big(\frac{\log \log n}{n}\big)}^{\frac{1}{2}}\Big)+ O(h_{n})\quad a.s.
\end{eqnarray}
Clearly by \eqref{fad22}, one has
\begin{eqnarray*}
h^{-\frac{1}{2}}_{n}{\bigg(\int_{0}^{T} {\Big|f_n(x)-f(x)\Big|}^{\frac{(p+2)}{2}} dx\bigg)}^{\frac{2}{(p+2)}}&=& O\Big({\big(\frac{\log \log n}{nh^{3}_{n}}\big)}^{\frac{1}{2}}\Big)\\
&\quad +&O\big(h^{\frac{1}{2}}_{n}\big)\quad a.s.
\end{eqnarray*}
Conditions $h_n\rightarrow0$\quad and\quad$\frac{\log \log n}{nh^{3}_{n}}\rightarrow 0,$ conclude that
\begin{eqnarray*}
h^{-\frac{1}{2}}_{n}\big|\hat{m}(p)-m(p)\big|\longrightarrow 0\quad a.s.
\end{eqnarray*}
The proof of ${\hat{\sigma}}^{2}(p)\longrightarrow {\sigma}^{2}(p)\quad a.s.$ is similar and will not be given.
\end{proof}

\end{document}